\documentclass[a4paper,10pt]{amsart}
\usepackage{textcomp}
\usepackage{amssymb, amsmath, amsthm}
\usepackage{xypic}
\usepackage{enumerate}
\xyoption{all}

\theoremstyle{theorem}
\newtheorem{theorem}{Theorem}[section]
\newtheorem{lemma}[theorem]{Lemma}
\newtheorem{proposition}[theorem]{Proposition}
\newtheorem{corollary}[theorem]{Corollary}

\theoremstyle{definition}
\newtheorem{definition}[theorem]{Definition}
\newtheorem{example}[theorem]{Example}

\theoremstyle{remark}

\def\Z{{\mathbf  {Z}}}

\def\a{{\textrm {Ann}}}
\def\s{{\textrm {Supp}}}
\def\o{\overline}

\def\C{{\mathcal  {C}}}
\def\ra{\rightarrow}
\def\A{{\mathcal {A}}}
\def\Q{{\mathbb  {Q}}}

\begin{document}

\title {Properties of cellular classes of chain complexes} 
\author[Kiessling] 
{Jonas Kiessling} 
 
\address{Kiessling: Department of Mathematics\\ 
Kungliga Tekniska H\"ogskolan \\ 
S -- 100 44 Stockholm\\ 
SWEDEN} 
 
\email{jonkie@kth.se} 
 
\thanks{Research supported by grant KAW 2005.0098 from 
the Knut and Alice Wallenberg Foundation.}


\begin{abstract} 
In this paper we prove certain properties of cellular and acyclic classes of chain complexes of modules over a commutative Noetherian ring. In particular we show that if $X$ is finite and belongs to some cellular class $\C$ then $\Sigma^n H_n X$ also belongs to $\C$, for every $n$.
\end{abstract}

\maketitle

\section{Introduction}
The aim of this paper is to study cellular classes of chain complexes in the derived category $D(R)$ of a commutative and Noetherian ring $R$. A collection $\C$ is called cellular if it is closed under arbitrary direct sums and, for any exact triangle $X\rightarrow Y\rightarrow Z\rightarrow \Sigma X$ with $X$ and $Y$ in $\mathcal C$, the cone $Z$ is also in $\mathcal C$. Cellular classes lead to the following relation between complexes: $Y$ is said to be {\em $X$-cellular}  if $Y$ belongs to any cellular class containing $X$. We denote this relation by $Y \gg X$. The collection $\mathcal{C}(X)$ of all $X$-cellular complexes is a cellular class. In fact it is the smallest cellular class containing $X$.

The idea of studying cellular classes comes from the work of Dror-Farjoun (\cite{MR1392221}). Dror-Farjoun uses cellular classes to set up a framework for doing unstable homotopy theory in which many traditional theorems have a common generalization and explanation. Classification of cellular classes of finite CW-complexes however turned out to be out of reach. One motivation for this paper was to investigate if such a classification can be obtained in a more algebraic setting.

The first aim of the paper is to derive some algebraic properties of the relation "$\gg$" between finite complexes, for us a finite complex is a chain complex $X$ such that the $R$-module $\oplus H_i X$ is finitely generated. In particular we wish to understand this relation between finitely generated modules (a module is a chain complex concentrated in degree zero). We show (see Corollary \ref{cor:modules}) that, for finitely generated modules $N$ and $M$,  $N\gg M$ if and only if $N$ is a quotient of a direct sum of $M$. Similar implication holds more generally. Assume that  $X$ is a finite complex. In Corollary \ref{cor:homepi} we show that if $f:X\ra Y$  induces an epimorphism on all homology groups, then $Y$ is $X$-cellular. The reverse implication does not hold.  Unlike for modules, cellular relations between finite chain complexes is hard to describe explicitly. An illuminating example of the complexity is to let $R$ equal the ring $\Z/(4)$, where the cellular relation can be used to distinguish all indecomposable perfect chain complexes in $D({\mathbf Z}/(4))$ (see \cite{Kiessling}). 

One reason why cellular relations between chain complexes is hard to describe is the fact that $\C(X)$ does not only depend on the homology of $X$. It is true that if  $X$ is finite, then, for any $i$, $\Sigma^{i}H_{i}(X)$ is $X$-cellular (see \ref{theorem:main}). However it is not true in general that  $X$ is $\oplus_{i}\Sigma^{i}H_{i}(X)$-cellular. 
 
The second aim of the paper is to describe a weaker relation that for finite complexes only depends on the homology. We say that a class $\mathcal A$ is acyclic if it is cellular and it is closed under extensions: for any exact triangle $X\rightarrow Y\rightarrow Z\rightarrow \Sigma X$ with $X$ and $Z$ in $\mathcal A$, $Y$ is also in $\mathcal A$.  Acyclic classes lead to the following relation between complexes: $Y$ is said to be $X$-{\em acyclic} if $Y$ belongs to all acyclic classes containing $X$. We denote this relation by $Y>X$. The collection $\mathcal{A}(X)$ of all $X$-acyclic complexes is an acyclic class. It is the smallest acyclic class containing $X$.

Original motivation for acyclicity comes again from topology. For Spectra and topological spaces the acyclic relation is well understood by the work of Hopkins-Devinatz-Smith (see \cite{MR960945}, \cite{MR1652975}) and Bousfield (\cite{MR1397720}). A decade long effort culminated in the classification of acyclic classes of finite spectra (\cite{MR960945}, \cite{MR1652975}).

As for $\gg$ we wish to describe the relation $>$ in algebraic terms. It turns out that this can be done for finite chain complexes. The main result (see Corollary \ref{cor:whenacyclic}) is that if $X$ and $Y$ are finite then $Y>X$ if and only if, for every $k$, $\s(H_k Y) \subset \s( \oplus_{i \leq k} H_i X)$ (we let $\s(M)$ denote the support of $M$). In particular the relation $>$ depends only on the homology. 

At the end of this paper we discuss an invariant of acyclic classes originally defined by Stanley in \cite{Staney}. To an acyclic class $\A$ we associate an increasing function $\phi_{\A}$ from $\Z$ to the set of all specialization closed subsets of Spec $R$. We prove that two acyclic classes $\A$ and $\A'$ give rise to the same function if and only if they contain the same finite chain complexes. See Corollary \ref{cor:stanley}. 

There is a close connection between the two relations introduced above (cf. Proposition \ref{prop:acycliccell}). A systematic study of the relations between cellularity and acyclicity for simplicial sets can be found in \cite{MR1408539}. The proof of Proposition \ref{prop:acycliccell} is also inspired by the techniques used in \cite{MR1408539}.

The results mentioned above should be compared with the work of Neeman (see \cite{MR1174255}). Neeman classifies certain acyclic classes called \emph{localizing}. An acyclic class is localizing if it is closed under desuspension. He shows that localizing classes correspond to (arbitrary) sets of prime ideals in $R$. A localizing class contains a finite chain complex $X$ if and only if the corresponding set of prime ideals contains the support of the homology of $X$. We can therefor regard the classification of the relation $>$ as a generalization of Neeman's result concerning finite complexes. 

Recently there has been an increased interest in cellular-type relations between algebraic objects. Far reaching results on cellular classes and cellular covers of groups are obtained in \cite{MR2269828}, \cite{MR2357479}, \cite{Wojtek} and \cite{Wojtek2}. Acyclic classes of chain complexes and their relation to $t$-structures are studied in \cite{Staney} and \cite{Alonso}. There is nowadays a whole industry concerned with support-systems and thick subcategories in various triangulated categories, see e.g. \cite{MR1450996} and \cite{Beson}. The \emph{support} of an object $X$ in a triangulated category $\mathcal{T}$ is the smallest full triangulated subcategory of $\mathcal{T}$ containing $X$ that is closed under direct sums (also called the \emph{thick subcategory generated by} $X$). An analogous theory for modules over a commutative ring is developed in \cite{Krause}. Although cellular classes do not correspond directly to support-systems, there are most likely connections between the two concepts that deserve to be explored further.       

The first person to study acyclic classes of chain complexes was Stanley in his paper \cite{Staney} on invariants of $t$-structures. Stanley proved that when restricted to the derived category of finite chain complexes the acyclic classes provide a complete invariant of $t$-structures. 

There is an overlap between this paper and that of Stanley. The results in this paper on the relation $>$ are present also in \cite{Staney}. However the focus of this paper is different, our aim is to understand both of the relations $\gg$ and $>$ and how they relate to each other. 
 
\subsection*{Acknowledgements}
The author would like to thank Wojciech Chach{\'o}lski for his valuable ideas, assistance and encouragement. The research for part of this paper was carried out while the author was a visiting researcher at the Thematic Program on Geometric Applications of Homotopy Theory at the Fields Institute for Mathematics, Toronto.

\section{Notation and Conventions}
We let $R$ denote a commutative and Noetherian ring.
For an $R$-module $M$, $\text{Supp}(M)$ denotes the support of $M$, that is the set of prime ideals $p \subset R$ for which the localization at this prime $M_p$ is non-zero. Recall that if $N$ and $M$ are $R-$modules and $N$ is a submodule or a quotient of $M$ then $\s(N) \subset \s(M)$. If $M$ and $N$ are finitely generated and the support of $N$ is contained in the support of $M$ then there are non-trivial homomorphisms $M \rightarrow N$ (see for instance in \cite{MR1251956} p.9).
 
The derived category of $R$ is denoted by $D(R)$. Recall that $D(R)$ is obtained from the category of chain complexes by formally inverting all quasi isomorphisms. We use the homological grading: the differential lowers the degree by 1. A chain complex is said to be \emph{finite} if the module $\oplus_{i \in \Z} H_i X$ is finitely generated. If $X$ is a chain complex then the $R$-module in degree $k$ is denoted by $X_k$. A chain complex is said to be \emph{non-negative} if it is zero in negative degrees. More generally it is \emph{bounded below (above)} if it vanishes in degrees small (big) enough. A chain complex is bounded if it is bounded below and above. A bounded complex of finitely generated projective modules is called \emph{perfect}. 

In general not every morphism in $D(R)$ comes from a morphism of chain complexes. However if $X$ is a bounded below chain complex of projective modules, then every morphism $X \rightarrow Y$ in $D(R)$ can be lifted to a morphism of chain complexes $X \rightarrow Y$. This lift is unique up to homotopy. 

The symbol $\Sigma:D(R) \rightarrow D(R)$ denotes the suspension functor and $\otimes$ the left derived tensor product. The category $D(R)$ is  triangulated. In particular we have a collection of exact triangles and any $f:X \rightarrow Y$ can be completed to an exact triangle 
\[
\xymatrix{X \ar^{f}[r] &  Y \ar[r] & C_f \ar[r] & \Sigma X}
\]
We refer to $C_f$ as the cone of $f$ and to $\Sigma^{-1} C_f$ as the fiber of $f$. The cone $C_f$ is unique up to a non-unique isomorphism. For a more detailed account of $D(R)$ see for instance \cite{MR1269324}. 

Suppose that we are given a (possibly transfinite) directed system in $D(R)$:
\[
F = \xymatrix{X_0 \ar[r] & X_1 \ar[r] & \ldots \ar[r] & X_{\alpha} \ar[r] & X_{\alpha +1} \ar[r] & \ldots}
\]
The \emph{homotopy colimit} of this directed system is defined by the exact triangle:
\[ 
\xymatrix{\oplus_{i} X_i \ar^{1-\textrm{shift}}[r] & \oplus_{i} X_i \ar[r] & \textrm{Hocolim } F \ar[r] & \oplus_{i} \Sigma X_i}
\]
It is defined only up to a non-unique isomorphism. See \cite{MR1812507} or \cite{MR1214458} for more details.

The abelian category of $R$-modules is embedded in $D(R)$ as a the full subcategory consisting of those chain complexes which are concentrated in degree zero. We will therefor freely confuse modules with chain complexes concentrated in degree zero.


\section{A useful proposition} \label{section:support}
In this section we show how we in some cases can construct interesting maps from a chain complex to its homology.

If $X$ is chain complex such that $X_i=0$ for $i<k$, then the sequence of the zero homomorphisms in dimensions different from $k$ and the quotient homomorphism in dimension $k$ constitute a morphism of chain complexes  $X\ra \Sigma^{k}H_k(X)$ inducing an isomorphism on  $H_k$. It follows that if $X$ is a chain complex such that $H_i(X)=0$ for $i<k$, then there is a morphism $f:X\ra  \Sigma^{k}H_k(X)$ in $D(R)$ which induces an isomorphism on $H_k$. More generally:

\begin{proposition} \label{prop:useful}
Let $X$ be a bounded below chain complex with $H_i(X)$ finitely generated for every $i$. Let $p$ be a prime ideal such that $p\in \text{\rm Supp}(H_k(X))\setminus  \text{\rm Supp}\big(\oplus_{i<k}H_i(X)\big)$. Then there is a morphism $f:X\ra \Sigma^{k}H_k(X)$ in $D(R)$, for which the localization $f_p:X_p\ra\Sigma^{k}H_k(X)_p $ induces an isomorphism on $H_k$.
\end{proposition}
\begin{proof}
The assumption $p\in \text{\rm Supp}(H_k(X))\setminus  \text{\rm Supp}\big(\oplus_{i<k}H_i(X)\big)$ implies that $H_i(X_p)=0$ if $i<k$. Thus  there is $g:X_p\ra \Sigma^k H_k(X)_p$ inducing an isomorphism on $H_k$. It remains to show that $g$ is the localization of some $f:X\ra \Sigma^k H_k(X)$. In the case $X$ is perfect this follows from the fact that $\text{Hom}_{D(R)}(X,\Sigma^k H_k(X)) \otimes R_p \cong \text{Hom}_{D(R_p)}(X_p,\Sigma^k H_k(X)_p)$. In the general case, since $X$ is bounded below with finitely generated homology, it is isomorphic in $D(R)$ to a bounded below chain complex of finitely generated projective modules $Y$. Its sub-complex $Y'$, given by:
\[
Y'_i=\begin{cases}
Y_i &\text{ if } i\leq k+1\\
0&\text{ if } i> k+1
\end{cases}
\]
is then perfect. The composition $Y'\subset Y\cong X\ra X_p\xrightarrow{g} \Sigma^k H_k(X))_p$ therefore factors through $f':Y'\ra  \Sigma^k H_k(X)$.  Note however that any morphism of chain complexes $f':Y'\ra  \Sigma^k H_k(X)$ is a restriction of  $f:Y\ra   \Sigma^k H_k(X)$. This $f$ satisfies then the requirements of the proposition.
\end{proof}


\section{Cellular Classes} \label{section:cellular}

We say that a collection $\C$ of chain complexes is closed under cones if given any exact triangle:
\[
\xymatrix{X \ar[r] & Y \ar[r] & Z \ar[r] & \Sigma X}
\]
such that $X$ and $Y$ belong to $\C$ then so does $Z$.

\begin{definition} 
A non-empty class of objects $\C \subset D(R)$ is called a \emph{cellular class} if it is closed under arbitrary direct sums and cones. 
\end{definition} 

Any cellular class $\C$ contains the zero object, it is closed under isomorphisms and it is closed under suspensions. To see this note that the zero fits into the triangle:
\[
\xymatrix{X \ar^1[r] & X \ar[r] & 0 \ar[r] & \Sigma X}
\]
and if $X$ and $Y$ are isomorphic then there is an exact triangle:
\[
\xymatrix{0 \ar[r] & X \ar[r] & Y \ar[r] & 0}
\]
Finally $\Sigma X$ is the cone of the map $X \rightarrow 0$. 

It follows from the definitions that cellular classes are closed under directed homotopy colimits. They are also closed under retracts. This is a consequence of the fact that if $i:Y \rightarrow X$ and $p:X \rightarrow Y$ are maps in $D(R)$ such that $ip$ is the identity on $Y$ then $Y$ is the homotopy colimit of the directed system (see \cite{MR1812507} p. 65):
\[
\xymatrix{X \ar^{ip}[r] & X \ar^{ip}[r] & X \ar^{ip}[r] & \ldots}
\]
Hence if $X$ belongs to a cellular class then so does $Y$.

Chain complexes that can be built out of $X$ using only sums and cones will belong to any cellular class containing $X$. The collection of all such complexes is itself a cellular class. This is the collection of all complexes cellular to $X$: 

\begin{definition}
Fix a chain complex $X$. We let $\C(X)$ denote the smallest cellular class containing $X$. Objects in $\C(X)$ are called $X-$cellular. If $Y$ is $X-$cellular then we write that $Y \gg X$. 
\end{definition}

Let $X$ be a chain complex and $I$ an index set. Cellular classes are closed under retracts and sums so any cellular class containing $X$ will contain $\oplus_I X$ and vice versa. Hence there is an equality $\C(\oplus_I X) = \C(X)$.

The cellular class $\C(R)$, the collection of all complexes cellular to the free module on one generator, is equal to all complexes with vanishing negative homology. To see this note that any such chain complex is isomorphic to a non-negative chain complex $X$ such that $X$ is free in each degree. Then write $X$ as the homotopy colimit of complexes $X^k$, $k \geq 0$, where $X^0$ is the module $X_0$ and inductively $X^{k+1}$ is the cone of $-\delta_{k+1}: \Sigma^k X_{k+1} \rightarrow X^k$.  
 
The relation $\gg$ is transitive. In other words, if $X$ is $Y$-cellular and $Y$ is $Z$-cellular then $X$ is $Z$-cellular. It also behaves well with respect to the tensor product:

\begin{lemma} \label{lemma:tensor}
Let $X$, $Y$ and $Z$ be chain complexes. If $Y$ is $X$-cellular, then $Y \otimes Z$ is $X \otimes Z$-cellular.
\end{lemma}
\begin{proof}
The tensor product with $Z$ takes exact triangles to exact triangles and preserves sums. Therefor the collection $\C$ of all chain complexes $Y$ such that $Y \otimes Z$ is $X \otimes Z$-cellular is a cellular class containing $X$. Hence $\C(X) \subset \C$ and we have proved the lemma.
\end{proof}

Fix an $R$-module $M$. Consider the collection $\C$ of all chain complexes $X$ that have vanishing negative homology and $H_0(X)$ is a quotient of a sum of $M$. The class $\C$ is a cellular class containing $M$. So if $X$ is $M$-cellular then $H_0(X)$ must be a quotient of a sum of $M$. In general this is not a sufficient condition. However we will show that if $M$ and $N$ are finitely generated modules then $N \gg M$ if and only if $N$ is a quotient of a sum of $M$ (Corollary \ref{cor:modules}).

More generally one can ask if cellular classes are closed under arbitrary homological epimorphisms of chain complexes.

We can let $R$ be the ring of integers $\Z$ and $\C$ the collection of all chain complexes $X \in D(\Z)$ such that the canonical map $X \rightarrow X \otimes \Q$ is an isomorphism, that is the homology of $X$ is a $\Q$-vector space. Because tensoring preserves exact triangles and commutes with direct sums, $\C$ is a cellular class. It also contains $\Q$. However it does not contain $\Q/\Z$. Hence cellular classes are in general not closed under homological epimorphisms. 

It is true however that cellular classes are closed under certain homological epimorphisms. For instance:

\begin{lemma} \label{lemma:principal}
Let $X$ be a chain complex and $I \subset R$ an ideal. Fix an integer $k$. If there is a map
\[
\xymatrix{f:X \ar[r] & \Sigma^k R/I}
\]
such that $f$ induces a surjection on homology. Then $\Sigma^k R/I \gg X$.
\end{lemma}

\begin{proof}
The $R$-module $R/I$ is $R$-cellular. By lemma \ref{lemma:tensor} $X \otimes R/I$ is $X \otimes R = X$-cellular. Suppose that the map $f$ exists. Then $\Sigma^k R/I$ is a retract of $X \otimes \Sigma^k R/I$. The class $\C(X)$ is closed under retracts so the lemma follows. 
\end{proof}

Later we show that cellular classes are closed under homological epimorphisms of finite chain complexes (see Corollary \ref{cor:homepi}).

Cellar classes are \emph{not} closed under homological monomorphisms, even of finitely generated modules. This is illustrated in the example below.

\begin{example} \label{example:notcellular}
Let $k$ denote some field and let $R = k[X, Y]$ be the polynomial ring in two variables. We fix the ideal $(X, Y)$. Any element of $(X,Y)$ generates a free $k[X,Y]$-module since $k[X,Y]$ is a domain. Hence we can regard $k[X,Y]$ as a submodule of $(X,Y)$.   

The image of any homomorphism of $k[X,Y]$-modules from $(X,Y)$ to $k[X,Y]$ must lie in the ideal $(X,Y)$. Hence there is no epimorphism from any sum of $(X,Y)$ onto $k[X,Y]$. 

As we noted above a necessary condition for a module to be $(X,Y)$-cellular is to be a quotient of a sum of $(X,Y)$, so $k[X,Y]$ is not $(X,Y)$-cellular. 
\end{example}

\section{Acyclic Classes} \label{section:acyclic}

We say that a class $\A$ of objects in $D(R)$ is closed under extensions if given any exact triangle:
\[
\xymatrix{
X \ar[r] & Y \ar[r] & Z \ar[r] & \Sigma X }
\]
such that $X$ and $Z$ belong to $\A$ then also $Y$ belongs to $\A$. 

Only certain cellular classes are closed under extensions:

\begin{definition}
A cellular class $\A \subset D(R)$ is called an \emph{acyclic class} if it is closed under extensions. 
\end{definition}

If we fix $X$ then the intersection of all acyclic classes containing $X$ is an acyclic class. This is the class of $X$-acyclic complexes:

\begin{definition}
Let $X \in D(R)$ be a chain complex. The class $\A(X)$ is defined to be the smallest acyclic class containing $X$. The objects of $\A(X)$ are called $X-$acyclic. If $Y$ is $X-$acyclic then we write that $Y > X$.
\end{definition}

Acyclic classes are easier to classify then cellular classes. We will show that if $X$ is finite then $\A(X)$ depends only on the support of the homology of $X$ (see Corollary \ref{cor:homologyonly} and \ref{cor:whenacyclic}). We will also show that acyclic classes are closed under homological epimorphisms and homological monomorphisms between finite chain complexes (see Corollary \ref{cor:hommonoepi}).

The relation $>$ is transitive. It also behaves well with respect to the tensor product:

\begin{lemma} \label{lemma:tensor2}
Let $X$, $Y$ and $Z$ be chain complexes. If $Y>X$, then $Y \otimes Z > X \otimes Z$.
\end{lemma}
\begin{proof}
The proof goes as the proof of \ref{lemma:tensor}. 
\end{proof}

Since $\A(X)$ is a cellular class containing $X$ we have by definition $\C(X) \subset \A(X)$. So the relation $Y \gg X$ implies the relation $Y > X$. 

In general a cellular class need not be acyclic, hence $Y>X$ does not imply $Y \gg X$. However the proposition below indicates that knowledge about the class $\A(X)$ does have strong implications for $\C(X)$. We will make use of this proposition several times in the later chapters, see for instance the proofs of Corollary \ref{cor:modules} and Theorem \ref{theorem:main}. 

\begin{proposition} \label{prop:acycliccell}
Suppose that in an exact triangle:
\[
\xymatrix{Y \ar[r] & Z \ar[r] & W \ar[r] & \Sigma Y }
\]
the relations $Y>X$ and $Z \gg X$ hold. Then also $W \gg X$.
\end{proposition} 

\begin{proof}
Let $\A$ be the class of all chain complexes $Y$ such that given any $n \geq 0$ and any exact triangle
\[
\xymatrix{\Sigma^n Y \ar[r] & Z \ar[r] & W \ar[r] & \Sigma^{n+1} Y }
\]
with $Z \gg X$, then $W \gg X$.

We wish to show that $\A(X) \subset D$. For this it is enough to show that $D$ is an acyclic class containing $X$. 

By definition $D$ is closed under suspensions. Also $X \in D$  since cones preserves cellularity.

It remains to show that $D$ is closed under arbitrary sums, cones and extensions. We shall make use of the Octahedral Axiom (\cite{MR1269324} p.378 or \cite{MR1812507} p. 58).

We first show that it is closed under cones. Let
\[
\xymatrix{Y_1 \ar[r] & Y_2 \ar[r] & Y_3 \ar[r] & \Sigma Y_1}
\]
be any exact triangle such that $Y_1$, $Y_2 \in D$. Suppose that $Z$ is $X$-cellular and that we have a map $\Sigma^n Y_3 \rightarrow  Z$ for some $n \geq 0$. We need to show that the cone $C_3$ of this map, is $X-$cellular. For simplicity we may assume that $n=0$. By the Octahedral axiom the triple $Y_2 \rightarrow  Y_3 \rightarrow  Z$ can be completed into a commutative diagram:
\[
\xymatrix{
Y_2 \ar^1[d] \ar[r] & Y_3 \ar[d] \ar[r] & \Sigma Y_1 \ar[d] \ar[r] & \Sigma Y_2 \ar^1[d] \\
Y_2 \ar[r] \ar[d]   & Z \ar[d] \ar[r]   & C_2 \ar[d] \ar[r] & \Sigma Y_2 \ar[d] \\
0 \ar[d] \ar[r]     & C_3 \ar^1[r]  \ar[d]    & C_3 \ar[d] \ar[r] & 0 \ar[d] \\
\Sigma Y_2 \ar[r]   & \Sigma Y_3  \ar[r]      & \Sigma^2 Y_1 \ar[r] & \Sigma^2 Y_2 }
\]
Here $C_2$ denotes the cone of the composite map $Y_2 \rightarrow Z$. All rows and columns are exact triangles. The complex $C_2$ is $X-$cellular since $Y_2 \in D$. Moreover $\Sigma Y_1 \in D$ since $D$ is closed under suspensions. The third column shows that $C_3$ is $X-$cellular and we are done.

In a similar fashion we show that $D$ is closed under extensions.

We finally show that $D$ is closed under sums. Suppose that $\{Y_i\}_{\alpha}$ is a family of chain complexes such that $Y_i \in D$. Suppose further that $D$ is closed under sums over ordinals smaller then $\alpha$. We can assume that $\alpha$ is a limit ordinal. Let $Z$ be any $X-$cellular chain complex and fix a map $\Sigma^n (\oplus_{\alpha} Y_i) \rightarrow  Z$, for some $n \geq 0$ which we may assume to be 0. We denote the cone of this map by $C$. We wish to show that $C$ is $X-$cellular. 

Let $\o{Y}_k = \oplus_{i < k} Y_k$. For $k < k'$ there are canonical maps $\o{Y}_k \rightarrow \o{Y}_{k'}$. Let $C_k$ denote the cone of the composite $\o{Y}_k \rightarrow Z$. The canonical maps induce maps $C_k \rightarrow C_{k'}$ when $k < k'$. Moreover $C$ is isomorphic to the homotopy colimit of this directed system. By the induction hypothesis each $C_k$ is $X-$cellular.  A cellular class is closed under directed homotopy colimits so $C$ is $X-$cellular.

We conclude that $D$ is an acyclic class containing $X$. It follows that $\A(X) \subset D$ and we are done.
\end{proof}

An analogous to Proposition \ref{prop:acycliccell} was originally obtained in the topological setting by Dror-Farjoun (see \cite{MR1392221}).

As a first application we prove that the suspension of an $X$-acyclic complex is $X$-cellular.

\begin{corollary} \label{cor:suspacyclic}
Fix two chain complexes $X$ and $Y$. If $Y$ is $X$-acyclic then $\Sigma Y$ is $X-$cellular.
\end{corollary}

\begin{proof}
Assume that $Y>X$. From the discussion in Section \ref{section:cellular} know that $0 \gg X$. Now apply the previous proposition to the exact triangle
\[
\xymatrix{Y \ar[r] & 0 \ar[r] & \Sigma Y \ar[r] & \Sigma Y}
\]
\end{proof}

Let $X$ be a chain complex and $I$ an index set. Since acyclic classes are in particular cellular, they are closed under retracts and sums. It follows that any acyclic class containing $X$ will contain $\oplus_I X$ and vice versa. Hence $\A(X) = \A(\oplus_I X)$.

If we fix an $R$-module $M$ and consider the collection $\A$ of all chain complexes $X$ such that the support of the homology of $X$ is a subset of the support of $M$, then it is straightforward to verify that $\A$ is an acyclic class containing $M$, hence $\A(M) \subset \A$. In other words a necessary condition for $X$ to be $M$-acyclic is $\s(\oplus H_i X) \subset \s(M)$. Compare with \ref{theorem:modules} and \ref{cor:whenacyclic}.

\section{Cellularity and Acyclicity of Modules} \label{section:acyclicitymodules}

This section is dedicated to the proof of a theorem characterizing the cellular and acyclic relations between finitely generated $R$-modules.

Let $M$ and $N$ be modules. We noted in section \ref{section:acyclic} that if $N$ is $M$-acyclic then the support of $N$ must be a subset of the support of $M$. The main result in this section is that for finitely generated modules the converse is also true. That is: 

\begin{theorem} \label{theorem:modules}
Let $M$ and $N$ be finitely generated $R$-modules. Then $N$ is $M$-acyclic if and only if the support of $N$ is contained in the support of $M$.
\end{theorem}

Before proving this theorem we state and prove a corollary concerning cellularity.

\begin{corollary} \label{cor:modules}
Let $M$ and $N$ be finitely generated $R$-modules. Then $N$ is $M$-cellular if and only if $N$ is a quotient of a sum of $M$.
\end{corollary}
\begin{proof}
The condition is necessary (see section \ref{section:cellular}). 

Suppose that there is a surjection from a sum of $M$ onto $N$. Since $N$ is finitely generated we may suppose that this sum is finite. In fact we can assume that we have a surjection from $M$ onto $N$ since $M$ and any sum of $M$ generate the same cellular class.

We need to show that given two finitely generated modules $M$ and $N$ such that there is a surjection from $M$ onto $N$, then $N$ is $M$-cellular. Fix such a surjection $f$ and complete it into a triangle:
\[
\xymatrix{C \ar[r] & M \ar^f[r] & N \ar[r] & \Sigma C} 
\]
By Proposition \ref{prop:acycliccell} it is enough to show that $C$ is $M$-acyclic. The map $f$ is a surjection so from the associated long exact sequence of homology modules we see that $C$ is isomorphic to a finitely generated module. By Theorem \ref{theorem:modules} proving that the module $C$ is $M$-acyclic is equivalent to showing that the support of $C$ is contained in the support of $M$.  From the same long exact sequence we see that $C$ is actually a submodule of $M$ so its support is contained in that of $M$ and we are done.
\end{proof}

We begin the proof of \ref{theorem:modules} by establishing three lemmas.

\begin{lemma} \label{lemma:annihilator}
Let $M$ be any $R$-module and $I \subset R$ an ideal. Then $M$ is $R/I$-cellular if and only if $M$ is an $R/I$-module.
\end{lemma}
\begin{proof}
The condition is necessary (see section \ref{section:cellular}). 

To prove that it is also sufficient we let $F$ denote any projective resolution of $M$ as an $R/I$-module. It is clear that $F \in C(R/I)$. Since $M$ and $F$ are isomorphic, $M$ must also be $R/I$-cellular. 
\end{proof}   

A trivial consequence of the lemma is that $M$ is $R/\a(M)$-acyclic. The key point in the proof of Theorem \ref{theorem:modules} is that these two modules actually generate the same acyclic class if $M$ is finitely generated. The following lemma is a tool when dealing with finite sums of cyclic modules.

\begin{lemma} \label{lemma:sums}
Let $I$ and $J$ be two ideals. The following three modules generate the same acyclic class:
\begin{enumerate}
\item $R / J \oplus R / I$
\item $R / (I J)$
\item $R / (I \cap J)$
\end{enumerate}
\end{lemma}
\begin{proof}
It follows from Lemma \ref{lemma:annihilator} that the only thing we need to show is that $R/(I J)$ is $R/I \oplus R/J$-acyclic.

We write $R / (I J)$ as an extension:
\[
\xymatrix{I / (I J) \ar[r] & R / (I J) \ar[r] & R / I \ar[r] & \Sigma I/(IJ)}
\]
The module $I / (I J)$ is annihilated by $J$ so it is  $R/J$-cellular, in particular $R/I \oplus R/J$-acyclic.

Acyclic classes are closed under extensions so $R/(IJ)$ is $R/I \oplus R/J$-acyclic and we are done. 
\end{proof}

\begin{lemma} \label{lemma:suspension}
Suppose that $M$ is a finitely generated module and let $I$ denote its annihilator. Then for some $k$, $\Sigma^k R/I$ is $M$-acyclic.
\end{lemma}

\begin{proof}
We fix a finite set $\{m_1, \ldots, m_n\}$ of generators of $M$. Let $M_i \subset M$ denote the module generated by $\{m_1, \ldots, m_i\}$. Each quotient $M_i/M_{i-1}$ is isomorphic to $R/I_i$ for some ideal $I_i$. From Lemma \ref{lemma:principal} we conclude that $R/I_i$ is $M_i$-cellular for each $i$. This result together with the triangles
\[
\xymatrix{M_i \ar[r] & R/I_i \ar[r] & \Sigma M_{i-1} \ar[r] & \Sigma M_i}
\]
imply that for all $0<i \leq n$, $\Sigma M_{i-1}$ is $M_i$-cellular.

Inductively we conclude that $\Sigma^{n-1} M_1 \gg M$. However $M_1$ is isomorphic to $R/\a(m_1)$. The choice of $m_1$ was arbitrary so 
\[
\oplus_{0 \leq i \leq n} \Sigma^{n-1} R/\a(m_i) \gg M
\]

In light of Lemma \ref{lemma:sums}, $\Sigma^k R/ (\cap \a(m_i)) > M$. The equality \[\cap \a(m_i) = \a(M)\] completes the proof.
\end{proof}

\begin{proof}[Proof of Theorem \ref{theorem:modules}]
Fix two finitely generated modules $M$ and $N$. If $N$ is $M$-acyclic then the support of $N$ must be a subset of the support of $M$ as we already noted. Therefor the condition is necessary. 

To show that it is also sufficient our strategy is as follows; we first show that $R/\a(M)$ is $M$-acyclic. It then follows that ``being $M$-acyclic'' is closed under taking quotients of finitely generated modules. Finally we show that $N > M$.

Let $I_0$ denote the annihilator of $M$. We claim that $R/I_0$ is $M$-acyclic. Since $\s(R/I_0) = \s(M)$ there is a non-zero map $f_0$ from $M$ to $R/I_0$. If $f_0$ is not surjective then its cokernel is some module $R/I_1$ where $I_0 \subset I_1$. Note that $\s(R/I_1) \subset \s(M)$ so there is some non-zero map $f_1$ from $M$ to $R/I_1$. Inductively this defines a sequence of ideals
\[
\a(M) = I_0 \subset I_1 \subset I_2 \subset \ldots
\]
and a family of maps $\xymatrix{f_i:M \ar[r] & R/I_i}$ such that the cokernel of $f_i$ is equal to $R/I_{i+1}$.

The ring $R$ is Noetherian so for some $k$, $I_{k+1} = I_{k+2} = \ldots$. This is only possible if $f_k$ is surjective. From Lemma \ref{lemma:principal} it follows that $R/I_k$ is $M$-cellular.

For all $i$ we let $C_i$ denote the cone of $f_i$. This is a chain complex concentrated in degrees 0 and 1, moreover $H_0 C_i = R/I_{i+1}$. Its first homology is a submodule of $M$. It is therefor annihilated by $\a(M) = I_0$. By Lemma \ref{lemma:annihilator} $H_1 C_i$ is $R/I_0$-cellular for all $i$.

We next consider the following exact triangles:
\[
\begin{array}{cccccccc}
M & \rightarrow & R/I_i & \rightarrow & C_i & \rightarrow & \Sigma M & (1) \\
\Sigma H_1 C_i & \rightarrow & C_i & \rightarrow & R/I_{i+1} & \rightarrow & \Sigma^2 H_1 C_i & (2)
\end{array}
\]
We noted above that there is some $k$ such that $R/I_k$ is $M$-cellular, in particular $M$-acyclic. We also noted that $H_1 C_i$ is $R/I_0$-cellular for all $i$. It follows from $(2)$ with $i = k-1$ that $C_{k-1}$ is $M \oplus \Sigma R/I_0$-acyclic. This result in (1) with $i = k-1$ implies that $R/I_{k-1}$ is $M \oplus \Sigma R/I_0$-acyclic. By induction it follows that 
\[
R/I_0 > M \oplus \Sigma R/I_0
\]
Since the relation $>$ is transitive and preserved by suspensions (see section \ref{section:acyclic}) we get that for any $n \geq 0$:
\[
R/I_0 > M \oplus \Sigma^n R/I_0
\]
By Lemma \ref{lemma:suspension} there is some $k$ such that $\Sigma^k R/I_0$ is $M$-acyclic. Hence $R/I_0$ is $M$-acyclic. This proves the claim.

Suppose that $K$ is a quotient of a finitely generated module $L$. By the above we know that $R/\a(L) > L$. Since $K$ is a quotient of $L$ the annihilator of $K$ contains the annihilator of $L$. By lemma \ref{lemma:annihilator} $K \gg R/\a(L)$. This shows that $K > L$. In particular, if $L$ is $M$-acyclic then so is $K$. This shows that being $M$-acyclic is closed under taking quotients of finitely generated modules.  
 
We now assume that the support of the finitely generated module $N$ is contained in the support of $M$. We wish to show that $N$ is $M$-acyclic.

We say that a module $K$ is generated by $M$ if there is a surjection from a sum of $M$ onto $K$. Let $K_0 \subset N$ denote the largest submodule of $N$ generated by $M$, this is the image of the evaluation map Hom$(M,N) \otimes_R M \rightarrow N$. The module $K_0$ is non-zero since there are non-zero maps from $M$ to $N$. Since $K_0$ is a quotient of a sum of $M$ it is $M$-acyclic. We are done if $K_0 = N$. If not let $L_1$ denote the quotient $N / K_0$. Define inductively $K_i \subset L_i$ as the largest submodule generated by $M$. They are all non-zero as long as $L_i$ is non-zero. Let $L_{i+1}$ be the quotient $L_i / K_i$. Since $N$ is a finitely generated module there is some $k$ such that $K_k \cong L_k$. 

As quotients of a finite sum of $M$, all the $K_i$'s are all $M$-acyclic. Since $K_k \cong L_k$, $L_k > M$. Now it follows by induction that $N$ is $M-$acyclic. 

This concludes the proof of the theorem.

\end{proof}

\section{Cellularity and Acyclicity of Chain Complexes}

The main result of this section is that if $X$ is a finite chain complex then $X$ builds its homology: $\Sigma^k H_k X$ is $X$-cellular for every $k$. One consequence of this theorem is that we can determine the relation $>$ between all finite chain complexes. It also implies that cellular classes are closed under homological epimorphisms between finite chain complexes. We conclude this section by reproving a classification by Stanley (\cite{Staney}) of all acyclic classes in $D^f(R)$.

Recall that a chain complex $X$ is called finite if the module $\oplus_{i \in \Z} H_i X$ is finitely generated. The full subcategory of $D(R)$ generated by all finite chain complexes is denoted by $D^f(R)$. 

We begin by showing that any finite chain complex is in the acyclic class generated by its homology.

\begin{proposition} \label{lemma:homology}
Let $X$ be any finite chain complex. Then
\[
X > \oplus_{i \in \Z} \Sigma^i H_i X
\]
\end{proposition}
\begin{proof}
Any finite chain complex is isomorphic to a bounded chain complex. Suppose that $X$ is bounded. There is a finite filtration of $X$:
\[
\ldots \subset F_i X \subset F_{i-1} X \subset \ldots X
\]
such that each quotient is isomorphic to the shifted module:
\[
F_i X /F_{i-1} X \cong \Sigma^i H_i X
\]
In other words $X$ can be written as an extension of its homology modules. Acyclic classes are closed under extensions so the lemma follows.
\end{proof}

We now state the main result of this paper. 

\begin{theorem} \label{theorem:main}
Let $X$ be a finite chain complex. Then for every $k$, the chain complex $\Sigma^k H_k X$ is $X$-cellular.
\end{theorem}

As a first corollary we see that the acyclic class of a finite chain complex depends only on its homology.

\begin{corollary} \label{cor:homologyonly}
If $X$ is a finite chain complex then:
\[
\A(X) = \A(\oplus_{i \in \Z} \Sigma^i H_i X)
\]
\end{corollary}
\begin{proof}
Combine Theorem \ref{theorem:main} with Proposition \ref{lemma:homology}.
\end{proof}

It is \emph{not} true that $\C(X)$ only depends on the homology of $X$, see for instance \cite{Kiessling}.

We can also describe the acyclicity relations between all finite chain complexes.

\begin{corollary} \label{cor:whenacyclic}
Fix two finite chain complexes $X$ and $Y$. Then $Y$ is $X$-acyclic if and only if for every $k$:
\[
\s(H_k Y) \subset \s(\oplus_{i \leq k} H_i X)
\]
\end{corollary}

\begin{proof}
Combine Corollary \ref{cor:homologyonly} with Theorem \ref{theorem:modules}.
\end{proof}

An other consequence is that cellular classes are closed under homological epimorphisms, at least between finite chain complexes.
\begin{corollary} \label{cor:homepi}
Let $X$ and $Y$ be two finite chain complexes. If there is a map $f:X \rightarrow Y$ such that for every $k$, $H_k f$ is an epimorphism, then $Y$ is $X$-cellular.
\end{corollary}
\begin{proof}
We complete the map $f$ into an exact triangle
\[
\xymatrix{ X \ar^f[r] & Y \ar[r] & C \ar[r] & \Sigma X}
\]
From the induced long exact sequence of homology and the assumption that $H_k f$ is an epimorphism, for every $k$, we see that $H_k C$ is a submodule of the finitely generated module $H_{k-1} X$. By Corollary \ref{cor:whenacyclic} $C > \Sigma X$. From Proposition \ref{prop:acycliccell} we deduce that $Y \gg X$. The proof is complete.
\end{proof}

Hence also acyclic classes are closed under homological epimorphisms. They are also closed under homological subcomplexes:
\begin{corollary} \label{cor:hommonoepi}
Let $X$ and $Y$ be finite chain complexes. If there is a map $f:X \rightarrow Y$ such that for every $k$, $H_k f$ is an epimorphism then $Y$ is $X$-acyclic. If instead $H_k f$ is a monomorphism for every $k$, then $X$ is $Y$-acyclic.
\end{corollary}
\begin{proof}
The first statement is a particular case of Corollary \ref{cor:homepi}. If instead $H_k f$ is a monomorphism for all $k$ then $\s(H_k X) \subset \s(H_k Y)$ and by Corollary \ref{cor:whenacyclic}, $X>Y$.
\end{proof}

Having established some consequences of Theorem \ref{theorem:main} we now turn to its proof. First a lemma.

\begin{lemma} \label{lemma:preliminary}
Let $X \in D^f(R)$. Then, for each $k$, there is some $n_k \geq k$ such that $\Sigma^{n_k} H_k X$ is $X-$cellular.
\end{lemma}
\begin{proof}
The idea of this proof is similar to that of Lemma \ref{lemma:suspension}.

First we assume that $X$ is non-negative and that $H_0 X \neq 0$. 

$H_0 X$ is generated by a finite set of elements $\{x_1, \ldots, x_n\}$. Let $M_i \subset H_0 X$ be the submodule generated by $\{x_1, \ldots, x_i\}$. We define $M_0 = 0$. Then the $M_i$'s filter $H_0 X$:
\[
0 = M_0 \subset M_1 \subset \ldots \subset M_n = H_0 X
\]
Each quotient $M_i/M_{i-1}$ is isomorphic to a cyclic module $R/I_i$ for some ideal $I_i$. Moreover $I_1 = \a(x_1)$.

There is a map $f:X \rightarrow H_0 X$ such that $f$ induces an isomorphism on $H_0$. Put $X_n = X$ and $f_n = f$. For $0 < i \leq n$ we let $X_{i-1}$ be the fiber of $f_i:X_i \rightarrow M_i/M_{i-1} = R/I_i$  and let $f_{i-1}:X_{i-1} \rightarrow N_{i-1}$ be the induced map. For each $0 < i \leq n$ we have a map of two exact triangles:
\[
\xymatrix{
X_{i-1} \ar[d] \ar[r] & X_i \ar[d] \ar[r] & R/I_i \ar[d] \ar[r] & \Sigma X_{i-1} \ar[d] \\
M_{i-1} \ar[r] & M_i \ar[r] & R/I_i \ar[r] & \Sigma M_{i-1} }
\]

The maps $X_i \rightarrow R/I_i$ induce a surjection homology so by Lemma \ref{lemma:principal} $R/I_i \gg X_i$, for every $i$. Together with the above exact triangles we see that $\Sigma X_{i-1} \gg X_i$ for every $i$.

From transitivity of the relation $\gg$ we deduce that $\Sigma^{n-1} X_1 \gg X$. However $X_1$ is isomorphic to $R/\a(x_1)$. The indexing of the generators was arbitrary so in fact:
\[
\oplus_{1 \leq i \leq n} \Sigma^{n-1} R/\a(m_i) \gg X
\]
By Lemma \ref{lemma:sums} this implies
\[
\Sigma^{n-1} R/\a(H_0 X) > X
\]
Now using Theorem \ref{theorem:modules}, we see that in fact $\Sigma^{n-1} H_0 X > X$. Hence by \ref{cor:suspacyclic}, $\Sigma^n H_0 X \gg X$.

Let $k$ be such that $H_i X = 0$ for $i < k$. The above discussion shows that there is some $n_0$ such that $\Sigma^{n_0} H_k X$ is $X$-cellular. Way may assume that $X$ is bounded. Recall the filtration in the proof of Proposition \ref{lemma:homology}. The exact triangle 
\[
\xymatrix{X \ar[r] & \Sigma^k H_k X \ar[r] & \Sigma F_{k+1} X \ar[r] & X}
\]
shows that $\Sigma^{n_0+1} F_{k+1} X$ is $X$-cellular. The first non-zero homology of $F_{k+1} X$ is equal to $H_{k+1} X$. The lemma now follows from an induction over $k$.
\end{proof}
 
\begin{proof}[Proof of Theorem \ref{theorem:main}]

We show by induction over $k$ that for all $i \leq k$:
\[
\Sigma^i H_i X \gg X \oplus \bigoplus_{k+1 \leq j} \Sigma^j H_j X \quad (*)
\]

Assume that $(*)$ holds for all $k \leq n-1$, for some $n$. To complete the induction we need to show that $(*)$ holds when $k = n$. From $(*)$ with $k = n-1$ we see that:
\[
\Sigma^i H_i X \gg X \oplus \bigoplus_{n \leq j} \Sigma^j H_j X 
\]
If we can show that 
\[
\Sigma^n H_n X \gg X \oplus \bigoplus_{n+1 \leq j} \Sigma^j H_j X 
\]
then it follows that for $i < n$:
\[
\Sigma^i H_i X \gg X \oplus \Sigma^n H_n X \oplus \bigoplus_{n+1 \leq j} \Sigma^j H_j X \gg X \oplus \bigoplus_{n+1 \leq j} \Sigma^j H_j X 
\]
In other words, to complete the induction we only need to verify $(*)$ for $k = n$ and $i = n$. 

The $R$-module Hom$_{D(R)}(X, \Sigma^n H_n X)$ is finitely generated. Let $I$ denote a finite set of generators. There is a natural map $\oplus_I X \rightarrow \Sigma^n H_n X$ which we complete into an exact triangle:
\[
\xymatrix{\oplus_I X \ar[r] & \Sigma^n H_n X \ar[r] & C \ar[r] & \oplus_I \Sigma X & (**)}
\] 
The sequence $(**)$ has the property that any map $X \rightarrow \Sigma^n H_n X$ will factor through $\oplus_I X \rightarrow \Sigma^n H_n X$.

\begin{lemma}
Let $C$ be defined as above. Then: $C > \oplus_{j \in \Z} \Sigma^{j+1} H_j X$
\end{lemma}
\begin{proof}
Recall that by Proposition \ref{lemma:homology} $C > \oplus \Sigma^i H_i C$. Thus to prove the lemma it is enough to show that for each $i$, $\Sigma^i H_i C > \oplus_{j < i} \Sigma^{j+1} H_j X$. From the long exact sequence of homology associated to $(*)$ we deduce that for $i \neq n, n+1$ the homology of $C$ is $H_i C \cong \oplus_I H_{i-1} X$. Therefor $\Sigma^i H_i C > \Sigma^{i} H_{i-1} X$ for $i \neq n, n+1$. We also have an exact sequence:
\[
\begin{array}{l}
\xymatrix{0 \ar[r] & H_{n+1} C \ar[r] & \oplus_I H_n X \ar[r] & H_n X \ar[r] & } \\
\xymatrix{H_n C \ar[r] & \oplus_I H_{n-1} X \ar[r]& 0}
\end{array}
\]
We see that $H_{n+1} C$ is a submodule of a finite sum of $H_n X$ so its support is a subset of the support of $H_n X$. By Theorem \ref{theorem:modules} this implies that $H_{n+1} C > H_n X$.

It remains to show that $\Sigma^n H_n C > \oplus_{j<n} \Sigma^{j+1} H_j X$

Let $M$ denote the cokernel of the map $\oplus_I H_n X \rightarrow  H_n X$. 

The $R$-module $M$ is finitely generated and $H_n C$ is an extension of $\oplus_I H_{n-1} X$ by $M$. To prove the claim we need to show that $M$ is $\oplus_{j<n} H_j X$-acyclic. By Theorem \ref{theorem:modules} this is equivalent to:
\[
\s(M) \subset \s(\oplus_{j<n} H_j X)  
\]
Assume that this is not the case. Then there is a prime ideal $p \in \s(M)$ such that $p \notin \s(\oplus_{j<n} H_j X)$. By Proposition \ref{prop:useful} there is a map $f:X \rightarrow \Sigma^n H_nX$ such that at $p$ the map $f_p$ induces an isomorphism on $H_n$. The prime ideal $p$ belongs to the support of $M$ so $M_p$ is non-zero and the composition of $f$ with the quotient map $\Sigma^n H_n X \rightarrow \Sigma^n M$ is non-zero on $H_n$. Call this composition $h$. By the construction of $C$ the map $f$ will factor through $\oplus_I X \rightarrow \Sigma^n H_n X$ so $H_n(h)$ must be zero. This is a contradiction. Hence $\Sigma^n H_n C > \oplus_{j<n} \Sigma^{j+1} H_j X$ and we have proved the lemma.
\end{proof}

The lemma together with the induction hypothesis implies that: 
\[
C > \Sigma X \oplus \Sigma^{n+1} H_n X \oplus \bigoplus_{n<j} \Sigma^{j+1} H_j X 
\]
This relation and the exact triangle $(**)$ shows that:  
\[
\Sigma^n H_n X > \Sigma^{n+1} H_n X \oplus X \oplus \bigoplus_{n<j} \Sigma^{j+1} H_j X
\]
Since the relation $>$ is transitive and preserved by suspensions we get that for any $m \geq 1$:
\[
\Sigma^n H_n X > \Sigma^{n+m} H_n X \oplus X \oplus \bigoplus_{n<j} \Sigma^{j+1} H_j X
\]
For $m$ large enough $\Sigma^{n+m} H_n X > X$ so in fact we have shown that:
\[
\Sigma^n H_n X > X \oplus \bigoplus_{n<j} \Sigma^{j+1} H_j X
\]
Substituting back for $C$ this gives:
\[
C > \Sigma X \oplus \bigoplus_{n<j} \Sigma^{j+1} H_j X 
\]

Together with Proposition \ref{prop:acycliccell} this completes the induction and also the proof of the theorem.

\end{proof}

We conclude this section with a final application of Theorems \ref{theorem:modules} and \ref{theorem:main}. 

In \cite{MR1174255} Neeman classifies all localizing subcategories of $D(R)$. A localizing subcategory is a full subcategory generated by a cellular class closed under $\Sigma^{-1}$. He proves that localizing subcategories are in a 1-1 correspondence with arbitrary subsets of Spec $R$ (the set of all prime ideals in $R$). Under this correspondence a finite chain complex $X$ will belong to a localizing subcategory $\L$ if and only if the set of prime ideals corresponding to $\L$ contains the support of the homology of $X$.

It is natural to ask for a similar classification of all cellular or acyclic classes. It turns out that there are in general too many cellular classes. In \cite{Staney} Stanley shows that the class of all acyclic classes in $D(\Z)$ is proper. However Stanley also shows that if we restrict to acyclic classes in $D^f(R)$ then a classification is possible.

We now reprove Stanley's classification using the results of this paper. 

We first define what an acyclic class in $D^f(R)$ is.

\begin{definition}
An \emph{acyclic class in $D^f(R)$} is a collection $\A$ of finite chain complexes such that:
\[
\A = D^f(R) \cap \A'
\]
where $ \A'$ is an (ordinary) acyclic class.
\end{definition}

The acyclic class $ \A'$ is not part of the definition. There might be two different acyclic classes $ \A'$ and $ \A''$ which give rise to the same acyclic class in $D^f(R)$.  

We say that a subset $V \subset $ Spec $R$ is closed under specialization if whenever $V$ contains a prime ideal $p$ then $V$ contains all prime ideals containing $p$.

To any acyclic class $\A$ in $D^f(R)$, we associate the function:
\[
\phi_{\A}: \Z \rightarrow 
\left\{ \begin{array}{c} \textrm{subsets of Spec $R$} \\
\textrm{closed under specialization} \end{array} \right\}
\]
sending $i \in \Z$ to 
\[
\phi_{\A}(i) = \{ p \in \textrm{ Spec $R$ }| \exists X \in \A \textrm{ such that $p \in \s(H_k X)$} \}
\]

This set is closed under specialization so the function is well defined.

A function 
\[\phi: \Z \rightarrow 
\left\{ \begin{array}{c} \textrm{subsets of Spec $R$} \\
\textrm{closed under specialization} \end{array} \right\}
\]
such that $\phi(i) \subset \phi(i+1)$ for every $i$ is called \emph{increasing}. For example, the function $\phi_{\A}$ is increasing.

To any increasing function $\phi$ we associate the acyclic class:
\[
\A'_{\phi} = \{ X \in D(R) | \textrm{ for every $k$, } \s(H_k X) \subset \phi(k) \}
\]
We let $\A_{\phi}$ denote the acyclic class in $D^f(R)$ given by $\A'_{\phi} \cap D^f(R)$.

In light of Theorem \ref{theorem:modules} and Corollary \ref{cor:whenacyclic} it is straightforward to verify that the maps
\[
\begin{array}{lcl}
\A & \mapsto & \phi_{\A} \\
\phi  & \mapsto & \A_{\phi} 
\end{array}
\]
are inverse isomorphisms between the set of all acyclic classes in $D^f(R)$ and the set of all increasing functions.

In all we have proved:

\begin{corollary} \label{cor:stanley}
The set of all acyclic classes in $D^f(R)$ is isomorphic to the set of all increasing functions
\[
\Z \rightarrow 
\left\{ \begin{array}{c} \textrm{subsets of Spec $R$} \\
\textrm{closed under specialization} \end{array} \right\}
\]
\end{corollary}




\begin{thebibliography}{FGSS07}

\bibitem[AJS]{Alonso}
L.~Alonso, A.~Jeremias, and M~Saorin, \emph{Classifying compactly generated
  t-structures on the derived category of a noetherian ring},
  arXiv:0706.0499v1.

\bibitem[BCR97]{MR1450996}
D.~J. Benson, Jon~F. Carlson, and Jeremy Rickard, \emph{Thick subcategories of
  the stable module category}, Fund. Math. \textbf{153} (1997), no.~1, 59--80.

\bibitem[BH93]{MR1251956}
Winfried Bruns and J{\"u}rgen Herzog, \emph{Cohen-{M}acaulay rings}, Cambridge
  Studies in Advanced Mathematics, vol.~39, Cambridge University Press,
  Cambridge, 1993.

\bibitem[BIK]{Beson}
D.~Benson, S.~Iyengar, and H.~Krause, \emph{Local cohomology and support for
  triangulated categories}, Ann. Sci. Ecole Norm. Sup., to appear.

\bibitem[BN93]{MR1214458}
Marcel B{\"o}kstedt and Amnon Neeman, \emph{Homotopy limits in triangulated
  categories}, Compositio Math. \textbf{86} (1993), no.~2, 209--234.

\bibitem[Bou96]{MR1397720}
A.~K. Bousfield, \emph{Unstable localization and periodicity}, Algebraic
  topology: new trends in localization and periodicity (Sant Feliu de Gu\'\i
  xols, 1994), Progr. Math., vol. 136, Birkh\"auser, Basel, 1996, pp.~33--50.

\bibitem[CDFS]{Wojtek2}
W.~Chach{\'o}lski, E.~Damian, E.D. Farjoun, and Y.~Segev, \emph{The $a$-core
  and $a$-cover of a group}, Preprint.

\bibitem[CFGS]{Wojtek}
W.~Chach{\'o}lski, E.~D. Farjoun, R.~G{\"o}bel, and Y.~Segev, \emph{Cellular
  covers of divisible abelian groups}, Preprint.

\bibitem[Cha96]{MR1408539}
Wojciech Chach{\'o}lski, \emph{On the functors {$CW\sb A$} and {$P\sb A$}},
  Duke Math. J. \textbf{84} (1996), no.~3, 599--631.

\bibitem[DHS88]{MR960945}
Ethan~S. Devinatz, Michael~J. Hopkins, and Jeffrey~H. Smith, \emph{Nilpotence
  and stable homotopy theory. {I}}, Ann. of Math. (2) \textbf{128} (1988),
  no.~2, 207--241.

\bibitem[Far96]{MR1392221}
Emmanuel~Dror Farjoun, \emph{Cellular spaces, null spaces and homotopy
  localization}, Lecture Notes in Mathematics, vol. 1622, Springer-Verlag,
  Berlin, 1996.

\bibitem[FGS07]{MR2269828}
Emmanuel~Dror Farjoun, R{\"u}diger G{\"o}bel, and Yoav Segev, \emph{Cellular
  covers of groups}, J. Pure Appl. Algebra \textbf{208} (2007), no.~1, 61--76.

\bibitem[FGSS07]{MR2357479}
Emmanuel~D. Farjoun, R{\"u}diger G{\"o}bel, Yoav Segev, and Saharon Shelah,
  \emph{On kernels of cellular covers}, Groups Geom. Dyn. \textbf{1} (2007),
  no.~4, 409--419.

\bibitem[HS98]{MR1652975}
Michael~J. Hopkins and Jeffrey~H. Smith, \emph{Nilpotence and stable homotopy
  theory. {II}}, Ann. of Math. (2) \textbf{148} (1998), no.~1, 1--49.

\bibitem[Kie]{Kiessling}
Jonas Kiessling, \emph{Classification of certain cellular classes of chain
  complexes}, arXiv:0801.3904v1.

\bibitem[Kra]{Krause}
Henning Krause, \emph{Thick subcategories of modules over commutative rings},
  Math. Annalen, to appear.

\bibitem[Nee92]{MR1174255}
Amnon Neeman, \emph{The chromatic tower for {$D(R)$}}, Topology \textbf{31}
  (1992), no.~3, 519--532, With an appendix by Marcel B\"okstedt.

\bibitem[Nee01]{MR1812507}
\bysame, \emph{Triangulated categories}, Annals of Mathematics Studies, vol.
  148, Princeton University Press, Princeton, NJ, 2001.

\bibitem[Sta]{Staney}
Don Stanley, \emph{Invariants of t-structures and classification of nullity
  classes}, arXiv:math/0602252v1.

\bibitem[Wei94]{MR1269324}
Charles~A. Weibel, \emph{An introduction to homological algebra}, Cambridge
  Studies in Advanced Mathematics, vol.~38, Cambridge University Press,
  Cambridge, 1994.

\end{thebibliography}

\providecommand{\bysame}{\leavevmode\hbox to3em{\hrulefill}\thinspace}
\providecommand{\MR}{\relax\ifhmode\unskip\space\fi MR }
\providecommand{\MRhref}[2]{%
  \href{http://www.ams.org/mathscinet-getitem?mr=#1}{#2}
}
\providecommand{\href}[2]{#2}

\end{document}